\documentclass{amsart}
\usepackage{amsmath, amsthm, amssymb}
\usepackage{verbatim}
\usepackage{tikz}

\usepackage{url}
\usepackage{hyperref}

\renewcommand{\d}{\partial}

\newcommand{\vepsilon}{\varepsilon}
\newcommand{\vphi}{\varphi}

\newcommand{\cH}{\mathcal{H}}

\newcommand{\supp}{\mbox{supp }}

\newcommand{\bN}{\mathbb{N}}

\newcommand{\bR}{\mathbb{R}}
\newcommand{\bC}{\mathbb{C}}

\newcommand{\vpi}{\varpi}

\newtheorem{thm}{Theorem}
\newtheorem{prop}[thm]{Proposition}
\newtheorem{lem}[thm]{Lemma}

\theoremstyle{definition}
\newtheorem{defn}[thm]{Definition}

\numberwithin{thm}{section}
\numberwithin{equation}{section}

\renewcommand{\[}{\begin{equation}}
\renewcommand{\]}{\end{equation}}

\newcommand{\wed}{\wedge}

\title[continuous subsolution problem and modulus of continuity]{A remark on the continuous subsolution problem for the complex Monge-Amp\`ere equation}

\author[S. Ko\l odziej, N.-C. Nguyen]{S\l awomir Ko\l odziej and Ngoc Cuong Nguyen}

\address{Faculty of Mathematics and Computer Science, Jagiellonian University 30-348 Krak\'ow, \L ojasiewicza 6, Poland}
\email{Slawomir.Kolodziej@im.uj.edu.pl}

\address{Faculty of Mathematics and Computer Science, Jagiellonian University 30-348 Krak\'ow, \L ojasiewicza 6, Poland; and Department of Mathematics, Center for Geometry and its Applications, Pohang University of Science and Technology, 37673, The Republic of Korea}
\email{Nguyen.Ngoc.Cuong@im.uj.edu.pl, \quad cuongnn@postech.ac.kr}


\begin{document}

\maketitle

\begin{center}
\em On the occasion of  L\^e V\u{a}n Thi\^em's centenary \rm
\end{center}

\bigskip
\bigskip
\bigskip

\begin{abstract}
We prove that if the modulus of continuity of a plurisubharmonic subsolution satisfies a Dini type condition then the Dirichlet problem for the complex Monge-Amp\`ere equation has the continuous solution. The modulus of continuity of the solution is also given if the right hand side is locally  dominated by  capacity.
\end{abstract}

\bigskip

\section{Introduction}

In this note we consider the  Dirichlet problem for the complex Monge-Amp\`ere equation in a strictly pseudoconvex domain $\Omega\subset \mathbb{C} ^n.$
Let $\psi$ be a continuous function on the boundary of $\Omega.$
  We look for the solution to the equation:
\[\label{eq:dirichlet-prob}
\begin{aligned}
&	u\in PSH(\Omega) \cap C^0(\bar\Omega), \\
&	(dd^c u)^n = d\mu, \\
&	u = \psi \quad \mbox{on } \d\Omega.
\end{aligned}\]

It was shown 
in \cite{ko96} that for the measures satisfying certain bound in terms of the Bedford-Taylor capacity \cite{BT2}
 the  Dirichlet problem has a (unique) solution. The precise statement is as follows.

Let $h : \mathbb R_+ \rightarrow (0, \infty ) $ be an increasing function such that
\[\notag
\label{eq:admissible}
	\int_1^\infty \frac{1}{x [h(x) ]^{\frac{1}{n}} }  \, dx < +\infty.
\]
We call such a function 
{\em admissible}. If $h$ is admissible, then so is $A  h$ for any number $A >0$.
Define
\[\notag
	F_h(x) = \frac{x}{h(x^{-\frac{1}{n}})}.
\]
Suppose that for such a function $F_h (x)$ a Borel measure $\mu $ satisfies

\begin{equation}
\label{eq:magrowth}
	\int_E d\mu \leq F_h( cap(E)),
\end{equation}
for any Borel set $E \subset \Omega$. Then, by  \cite{ko96} the Dirichlet problem  \eqref{eq:dirichlet-prob} has a solution.

This statement is useful as long as we can verify the condition \eqref{eq:magrowth}.
In particular if $\mu$ has density with respect to the Lebesgue measure in $L^p$, $p>1$
then this bound is satisfied \cite{ko96}. By the recent results in 
 \cite{Ng17, Ng18} if $\mu$ is bounded by the Monge-Amp\`ere measure of a H\"older continuous plurisubharmonic function $\vphi$:
\[\notag
\label{eq:sub}
	\mu \leq (dd^c\vphi)^n  \quad \mbox{in }\Omega ,
\]
then  \eqref{eq:magrowth} holds for a specific $h$, and consequently,  the Dirichlet problem \eqref{eq:dirichlet-prob} is solvable with H\"older continuous solution.
Our result in this paper says that we can considerably weaken the assumption on $\varphi$ and still get a continuous solution of the equation.

Let $\vpi (t):= \vpi(t;\vphi,\bar\Omega)$   denote the  modulus of continuity of $\vphi$ on $\bar\Omega$,  i.e, 
\[\notag
	\vpi(t)= \sup \left\{|\vphi(z) -\vphi(w)| : z,w \in \bar\Omega, \quad |z-w| \leq t\right\}.
\]
Thus
$
	|\vphi(z) - \vphi(w)| \leq \vpi (|z-w|)
$ for every  $ z,w\in \bar\Omega$.
Let us state the first result.

\begin{thm} 
\label{thm:sub}
Let $\vphi \in PSH(\Omega) \cap C^{0}(\bar\Omega), $ $\vphi =0$ on $\partial\Omega$. Assume that its modulus of continuity satisfies the Dini type condition
\[\label{eq:modulus-ass}
	\int_0^{1} \frac{[\vpi(t)]^\frac{1}{n}}{t |\log t|} dt <+\infty.
\]
If the measure $\mu$ satisfies $\mu \leq (dd^c\vphi)^n$ in $\Omega$, then the Dirichlet problem \eqref{eq:dirichlet-prob} admits a unique solution.

\end{thm}

Let us mention  in  this context that  it is still an open problem if a continuous subsolution $\vphi$ implies the
solvability of \eqref{eq:dirichlet-prob}.

The modulus of continuity of solution to the Dirichlet problem~\eqref{eq:dirichlet-prob} was obtained in \cite{BT76} for 
$\mu = fdV_{2n}$ with $f(x)$ being continuous on $\bar\Omega$. We also wish to study this problem for the measures which satisfy the inequality \eqref{eq:magrowth}. For simplicity we restrict ourselves to measures belonging to $\cH(\alpha,\Omega)$. In other words, we take the function $h(x) = C x^{n\alpha}$  for positive constants $C,\alpha>0$ in the inequality \eqref{eq:magrowth}.

 We introduce the following notion, which generalizes the one in  \cite{ko94}.
Consider a continuous increasing function $F_0:[0,\infty) \to [0,\infty)$ with $F(0)=0$. 

\begin{defn}
The measure $\mu$ is called  uniformly locally dominated by capacity with respect to $F_0$ if for every cube $I(z,r)=:I  \subset B_I:= B(z, 2r)  \subset \subset \Omega$ and for every set $E\subset I$,
\[\label{eq:u-vol-cap}
	\mu(E) \leq \mu(I) F_0\left(cap (E, B_I) \right).
\]
\end{defn}

According to \cite{ACKPZ} the Lebesgue measure $dV_{2n}$ satisfies this property with $F_0 = C_\alpha \exp (-\alpha/ x^{-1/n})$ for every $0< \alpha < 2n$. The case $F_0(x)= C x$ was considered in \cite{ko94}.  We refer the reader to \cite{BJZ} for more examples of measures satisfying this property.
Here is our second result.

\begin{thm} 
\label{thm:modulus}
Assume  $\mu \in \cH(\alpha,\Omega)$ with compact support and satisfying the condition \eqref{eq:u-vol-cap} for some $F_0$. Then, the modulus of continuity of the solution $u$ of the Dirichlet problem~\eqref{eq:dirichlet-prob} satisfies for $0< \delta < R_0$ and $2R_0 = \mbox{\rm dist} (\supp\mu, \d\Omega)>0$,
\[\notag
	\vpi(\delta;u,\Omega) \leq  \vpi(\delta;\psi,\d\Omega) + C \left[ \left(\log \frac{R_0}{\delta}\right)^{-\frac{1}{2}} + F_0 \left(\frac{C_0}{[\log (R_0/\delta)]^\frac{1}{2}}\right)\right]^{\alpha_1},
\]
where  the constants $C, \alpha_1$ depend only on $\alpha, \mu, \Omega$.
\end{thm}

\bigskip

{\bf Acknowledgement.} The first author was partially supported by NCN grant  
2017/27/B/ST1/01145. The second author was supported by  the NRF Grant 2011-0030044 (SRC-GAIA) of The Republic of Korea. He also would like to thank Kang-Tae Kim for encouragement and support.

\section{Proof of Theorem~\ref{thm:sub}}

In this section we will prove Theorem~\ref{thm:sub}. We need the following   lemma. The proof of this lemma is based on a similar idea as the one in \cite[Lemma~3.1]{KN18a} where the complex Hessian equation is considered. The difference is that we have much stronger volume-capacity inequality for the Monge-Amp\`ere equation.

\begin{lem} 
\label{lem:vol-cap-key}
Assume the measure $\mu$ is compactly supported. Fix $0< \alpha < 2n$ and $ \tau = \alpha/(2n+1)$. There exists a uniform constant $C$ such that for every compact set $K \subset \Omega$,
\[\label{eq:vol-cap-key}
	\mu(K) \leq C \left\{ \vpi\left( \exp \left(\frac{-\tau}{2[cap(K)]^\frac{1}{n}}\right) \right)  + \exp\left(\frac{2n\tau -\alpha}{2[cap(K)]^\frac{1}{n}} \right)\right\} \cdot cap(K)
\]
where $cap(K):= cap(K,\Omega).$
\end{lem}

\begin{proof} Fix a compact subset $K\subset\subset \Omega$. Without loss of generality we may assume that $K$ is regular (in the sense that its relative extremal function \cite{BT2}
is continuous) as $\mu$ is a Radon measure. Denote by $\vphi_\vepsilon$  the standard regularization of $\vphi$. We choose $\vepsilon>0$ so small that 
\[\notag
	\supp \mu  \subset \Omega'' \subset\subset \Omega' \subset  \Omega_\vepsilon  \subset \Omega,
\]
where $\Omega_\vepsilon = \{z\in \Omega: dist(z, \d\Omega) > \vepsilon\}$. Since for every $K \subset \Omega''$ we have
\[\notag
	cap(K, \Omega') \sim cap(K,\Omega)
\]
(up to a constant depending only on $\Omega, \Omega'$) in what follows we will write $cap (K)$ for either one of these capacities.
We have 
\[\notag
	0\leq \vphi_\vepsilon - \vphi \leq \vpi(\vepsilon) := \delta \quad
	\mbox{on } \Omega'.
\]
Let $u_K$ the relative extremal function for $K$ with respect to $\Omega'$. 
Consider the set $K' = \{ 3\delta u_K + \vphi_\vepsilon < \vphi - 2\delta\}$. Then, 
\[\label{eq:supp}
	K \subset K' \subset \left\{u_K < -\frac{1}{2} \right\} \subset \Omega'.
\]
Hence, by the comparison principle \cite{BT2},
\[\label{eq:compare}
	cap(K') \leq 2^n cap(K). 
\]
Note that 
\[\label{eq:bound-ep}
	dd^c \vphi_\vepsilon \leq \frac{C}{\vepsilon^2} \; dd^c |z|^2, \quad
	\|\vphi_\vepsilon + u_K\|_\infty=:M \leq \|\vphi\|_\infty +1.
\]
The comparison principle,  the bounds \eqref{eq:bound-ep} and the volume-capacity inequality from  \cite{ACKPZ} (in  the last inequality below) give us that 
\[\begin{aligned}
	\int_{K'} (dd^c \vphi )^n 
&\leq 	\int_{K'} (dd^c (3\delta u_K + \vphi_\vepsilon) )^n \\
&\leq		3\delta\int_{K'}   \left[dd^c (u_K + \vphi_\vepsilon)\right]^n  + \int_{K'} (dd^c \vphi_\vepsilon )^n \\
&\leq		3\delta   M^n cap(K') + C(\alpha) \vepsilon^{-2n} \exp\left(\frac{-\alpha}{[cap(K')]^\frac{1}{n}} \right) cap(K').
\end{aligned}\]
Choose $$\vepsilon = \exp\left(\frac{-\tau}{[cap(K')]^\frac{1}{n}} \right)$$ (we assume that $\vepsilon$ is so small that it satisfies \eqref{eq:supp}, otherwise the inequality \eqref{eq:vol-cap-key} holds true by increasing the constant) and plug in the formula for $\delta$ we get that
\[\notag
\begin{aligned}
	\mu(K) 
&\leq 	\int_{K'} (dd^c (\vphi) )^n \\
&\leq 	3 M^n \vpi\left( \exp \left(\frac{-\tau}{[cap(K')]^\frac{1}{n}}\right) \right) \cdot  cap(K')  \\
&\quad + C \exp\left(\frac{2n\tau -\alpha}{[cap(K')]^\frac{1}{n}} \right).
\end{aligned}\]
This combined with \eqref{eq:compare} gives the desired inequality.
\end{proof}

We are ready to finish the proof of the theorem. 
It follows from Lemma~\ref{lem:vol-cap-key} that a suitable function $h$ for the measure $\mu$ which satisfies  \eqref{eq:magrowth} is
\[\notag
	h (x)=  \frac{1}{C \vpi(\exp (-\tau x))}
\]
once we had 
\[\notag
	\int_1^\infty \frac{1}{x [h(x) ]^{\frac{1}{n}} }  \, dx < +\infty.
\]
By changing the variable $s= 1/x$, and then $t = e^{-\tau/s}$, this is equivalent to 
$$
	\int_0^{e^{-\tau}} \frac{\left[\vpi(t) \right]^\frac{1}{n}}{t |\log t|} dt <+\infty.
$$
The finiteness is guaranteed by  \eqref{eq:modulus-ass}.
Thus,  our assumption on the modulus of continuity $\vpi(t)$ implies that $h$ is admissible in the case of $\mu$ with compact support.
Then, by  \cite{ko96} the Dirichlet problem~\eqref{eq:dirichlet-prob} has a unique solution.

To deal with the general case consider the exhaustion of $\Omega$ by 
$$
E _j =\{ \vphi \leq  -1/j \}
$$
and define $\mu _j $ to be the restriction of $\mu$ to $E_j$. Denote by $u_j$ the solution
of  \eqref{eq:dirichlet-prob} with $\mu$ replaced by $\mu _j$. By the comparison principle
$$
u_j + \max (\vphi ,  -1/j ) \leq u \leq u_j ,
$$
and so the sequence $u_j$ tends to $u=\lim u_j$ uniformly which gives the continuity of $u$.
The proof is completed.

\section{the modulus of continuity of solutions} 

In this section we study the modulus of continuity of the solution of the Dirichlet problem with the right hand side in the class $\cH(\alpha,\Omega)$ (definition below)
under the additional condition that a given measure is locally  dominated by capacity.

Recall that a positive Borel measure $\mu$ belongs to  $\cH(\alpha,\Omega)$, $\alpha>0$, if there exists a uniform constant $C>0$ such that for every Borel set $E \subset \Omega$,
\[\notag
	\mu(E) \leq C \left[cap (E,\Omega) \right]^{1+ \alpha}.
\]
The following result \cite[Lemma~2]{ko94}  will be used in what follows.

\begin{lem}
\label{lem:ko94}
 Suppose $0< 3r< R$ and
$
	B(z,r) \subset B(z, R) \subset \subset \Omega.
$
Let  $v\in PSH(\Omega)$ be such that $-1 \leq v \leq 0$. Denote 
\[\notag
	E(\vepsilon, v, B(z,r)) := \{z \in B(z, r) : (1-\vepsilon)v \leq \sup_{B(z,r)} v\},
\]
where $\vepsilon \in (0,1)$. Then, there exists $C_0$ depending only on $n$ such that
\[\notag
	cap(E, B(z,2r)) \leq \frac{C_0}{\vepsilon \log (R/r)}.
\]
\end{lem}

\begin{proof} See Appendix.
\end{proof}

Let us proceed with the proof of Theorem~\ref{thm:modulus}.
Since $\mu \in \cH(\alpha,\Omega)$, according to \cite{ko96} we can solve the Dirichlet problem~\eqref{eq:dirichlet-prob} to obtain a  unique continuous solution $u$.
Define for $\delta>0$ small
\[\notag
	\Omega_\delta := \left\{z \in \Omega : dist(z, \d \Omega) > \delta\right\};
\]
and for $z\in \Omega_{\delta}$ we define
\[\notag
	u_\delta(z) := \sup_{|\zeta| \leq \delta} u(z+ \zeta).
\]
Thanks to the arguments in \cite[Lemma~2.11]{Ng17}  it is easy to see  that there exists $\delta_0>0$ such that
\[\label{eq:boundary-est-b}
	u_\delta(z) \leq u(z) +  \vpi(\delta;\psi,\d\Omega)
\]
for every $z \in \d\Omega_\delta$ and $0< \delta<\delta_0$.
Here we used the result of Bedford and Taylor \cite[Theorem~6.2]{BT76} (with minor modifications) to extend $\psi$ plurisubharmonically onto $\Omega$
so that  its modulus of continuity on $\bar\Omega$ is controlled by the one on the boundary. 
Therefore, for a suitable extension of $u_\delta$ to $\Omega$, using the stability estimate for measure in $\cH(\alpha,\Omega)$ as in \cite[Theorem~1.1]{GKZ08} (see also \cite[Proposition~2.10]{Ng17}) we get 

\begin{lem}
\label{lem:stability}
 There are uniform constants $C, \alpha_1$ depending only on $\Omega, \alpha, \mu$ such that
\[\notag	\sup_{\Omega_\delta} (u_\delta - u)  \leq \vpi(\delta; \psi, \d\Omega) + C \left(\int_{\Omega_\delta} (u_\delta -u) d\mu \right)^{\alpha_1}
\]
for every $0<\delta<\delta_0$.
\end{lem}
Thanks to this lemma we know that the right hand side  tends to zero as $\delta$ decreases to zero. We will use the property  "locally dominated by capacity" to obtain a quantitative bound via Lemma~\ref{lem:ko94}.

\begin{proof}[End of Proof of Theorem~\ref{thm:modulus}]
Let us denote the support of $\mu$ by $K$. Since $\|u\|_\infty$ is controlled by a contant $C = C(\alpha,\Omega, \mu)$, without loss of generality we may assume that 
\[\notag
	-1 \leq u \leq 0.
\]
Then for every $0<\vepsilon<1$
\[\label{eq:divide}
\begin{aligned}
	\int_{\Omega_\delta} (u_\delta -u) d\mu \leq \vepsilon \; \mu(\Omega) + \int_{\{u < u_\delta - \vepsilon\} \cap K} d\mu
\end{aligned}\]
We shall now estimate the second term on the right hand side.

Let us fix the notation that will be used later on. 
We may assume that $\Omega \subset\subset [0,1]^{2n}$. Let us write $z = (x^1, ...,x^{2n}) \in \bR^{2n}$ and denote the semi open cube centered at a point $z_0$ of diameter $2r$ by
\[\notag
	I(z_0,r):= \{z = (x^1, ..., x^{2n})\in \bC^n : -r \leq x^i - x_0^i<  r \; \forall i = 1,...,2n\}.
\]
Then, by the assumption $\mu$ satisfies for every cube $$I(z,r)=:I \subset B_I := B(z,2r) \subset\subset \Omega$$ and for every set $E \subset I$, 
\[ \label{eq:local-dominate}
	\mu (E) \leq  \mu (I(z,r)) F_0\left( cap (E, B_I) \right),
\]
where $F_0: [0,\infty] \to [0,\infty]$ is an increasing continuous function and $F_0 (0) =0$.

Consider the semi-open cube decomposition of $\Omega \subset\subset I_0:=[0,1)^{2n} \subset \bR^{2n}$ into $3^{2ns}$ congruent cubes of diameter $3^{-s} = 2\delta$, where $ s \in \bN$.  
Then 
\[ \label{eq:inclusion}
\{u <  u_\delta -\vepsilon\} \cap I_s  \subset \{z\in B_{I_s} : u< \sup_{B_{I_s}} u -\vepsilon\},
\]
where $I_s = I(z_s, \delta)$ and $B_{I_s} = B(z_s, 2\delta)$ for some $z_s\in I_0$.
Hence
\[\notag
	\int_{\{u <  u_\delta - \vepsilon\}} d\mu \leq  \sum_{I_s \cap K \neq \emptyset  }\int_{\{u< u_\delta -\vepsilon \} \cap I_s} d\mu.
\]
Using \eqref{eq:local-dominate}, \eqref{eq:inclusion}, and then applying Lemma~\ref{lem:ko94} for $r=2\delta$ and $R= 2R_0$, we have for $B_s : = B(z_s, 4\delta)$ corresponding to each cube $I_s$:
\[\label{eq:key-bound-u}
\begin{aligned}
	\int_{\{u< u_\delta - \vepsilon\} \cap I_s} d\mu 
&\leq 	 \mu(I_s) F_0(cap(E(\vepsilon, u, B_{I_s}), B_s))  \\
&\leq		 \mu(I_s) \; F_0 \left(\frac{C_0}{\vepsilon \log (R_0/\delta)}\right),	
\end{aligned}\]
where $2R_0 = \mbox{dist} (K, \d\Omega)$.
Therefore, combining the above inequalities, we get that
\[\notag
	\int_{\{u <  u_\delta - \vepsilon\}} d\mu \leq  \mu(\Omega) F_0 \left(\frac{C_0}{\vepsilon \log (R_0/\delta)}\right).
\]
We conclude from this  and Lemma~\ref{lem:stability} that
\[\notag
	\omega(\delta;u, \bar\Omega) \leq \sup_{\Omega_\delta} (u_\delta - u) \leq \vpi(\delta;\psi,\d\Omega) + C \left[\vepsilon + F_0 \left(\frac{C_0}{\vepsilon \log (R_0/\delta)}\right)\right]^{\alpha_1}.
\]
If we choose $\vepsilon = (\log R_0/\delta)^{-1/2}$ then Theorem~\ref{thm:modulus} follows. 
\end{proof}


\section{Appendix}

For the reader's convenience we give the details of the  proof of Lemma~\ref{lem:ko94}.
The following inequality is due to Alexander and Taylor \cite[Lemma~3.3]{AT84}.

\begin{lem} 
\label{lem:AT}
Let $B' = \{|z-z_0| <r \} \subset \subset B= \{|z-z_0| <R\}$ be two concentric balls centered at $z_0$ in $\bC^n$. Let $u \in PSH(B) \cap L^\infty(B)$ with $u<0$. There is a constant $C = C(n, \frac{R}{r})$ independent of $u$ such that
\[\notag
	\int_{B'} (dd^c u)^n \leq C |u(z_0)| \sup_{z\in B} |u(z)|^{n-1}.
\]
In particular, if $R/r = 3$ then the constant $C$ depends only on $n$.
\end{lem}

\begin{proof} Without loss of generality we may assume $z_0 \equiv 0$. Set $\rho:= (r+R)/2$ and $B(\rho)= \{|z-z_0| < \rho\}$. We use the B\l ocki inequality \cite{B} to get 
\[\notag\begin{aligned}
	\int_{B'} (dd^c u)^n 
&\leq		\frac{1}{(\rho^2 -r^2)^{n-1}}\int_{B(\rho)} |v|^{n-1} (dd^c u)^n \\
&\leq 	\frac{(n-1)! \|u\|_{B_\rho}^{n-1}}{(\rho^2 -r^2)^{n-1}} \int_{B(\rho)} dd^c u \wed \beta^{n-1},
\end{aligned}\]
where $v(z) = |z|^2 -\rho^2$ and $\beta := dd^c v = dd^c |z|^2$.
Next, by Jensen's formula:
\[\notag
	u(0) + N(\rho) = \frac{1}{\sigma_{n-1}} \int_{\{|\zeta|=1\}} u(\rho \zeta)  d\sigma(\zeta),
\]
where  $\sigma_{2n-1}$ is the area of the unit sphere, 
\[\notag
	N(\rho) = \int_0^\rho \frac{n(t)}{t^{2n-1}} dt
\]
and 
\[\notag
	n(t) = \frac{1}{\sigma_{n-1}} \int_{\{|z| \leq t\}} \Delta u(z) dV_{2n}(z) = a_n \int_{\{|z| \leq t\}} dd^c u \wed \beta^{n-1}.
\]
Since $n(t)/t^{n-2}$ is increasing, we have
\[\notag
	N(R) \geq \int_{\rho}^R \frac{n(t)}{t^{2n-1}} dt \geq \frac{n(\rho)}{\rho^{2n-2}} \log (R/\rho).
\]
From  $u<0$, it follows that  $N(R) < - u(0)$. 
Hence, 
\[\notag
	\int_{B_\rho} dd^cu \wed \beta^{n-1} \leq  \frac{n(\rho)}{a_n} \leq \frac{N(R)\rho^{n-2}}{\log(R/\rho)}  \leq  \frac{\rho^{2n-2} |u(0)|}{\log(R/\rho)}. 
\]
Combining  the above inequalities we get the desired estimate with the constant 
$$C = \frac{(n-1)!\rho^{2n-2}}{(\rho^2-r^2)^{n-1} \log (R/\rho)}.$$ 
If $R= 3r$, then $C$ is also independent of $r$.
\end{proof}

We are ready to prove Lemma~\ref{lem:ko94}. We shall reformulate it as in \cite[Lemma~2]{ko94} and follow the proof given there. 

\begin{lem} Denote for $\rho \geq 0$,
$
	B_\rho= \{|z-z_0|< e^\rho R_0\}.
$
Given $z_0 \in \Omega$ and two numbers $M>1$, $R_0>0$ such that
$
	B_M  \subset \subset\Omega,
$
and given $v\in PSH(\Omega)$ such that $-1<v<0$, denote by $E$  the set 
\[\notag
	E=E(\delta) = \{z \in B_0 : (1-\delta)v \leq \sup_{B_0} v\},
\]
where $\delta \in (0,1)$. Then, there exists $C_0$ depending only on $n$ such that
\[\notag
	cap(E, B_2) \leq \frac{C_0}{M\delta}.
\]
\end{lem}

\begin{proof} From the logarithmic convexity of the function $r \mapsto \sup_{|z-z_0|<r} v(z)$ it follows that for $z \in B_M \setminus B_0$ and $a_0:= \sup_{B_0} v$ we have
\[\notag
	v(z) \leq a_0 \left(1 - \frac{1}{M} \log\frac{|z-z_0|}{R_0}\right).
\]
Hence, 
\[\notag
	a := \sup_{B_2} v \leq a_0\left(1 - \frac{2}{M}\right).
\]
Let $u = u_{E,B_2}$ the relative extremal function of $E$ with respect to $B_2$. One has
\[\notag
	\frac{v-a}{a - a_0/(1-\delta)} \leq u.
\]
So, for some $z_1 \in \d B_0$ we have
\[\notag
	u(z_1) \geq \frac{a_0 - a}{a - a_0/(1-\delta)} \geq \frac{2(\delta -1)}{(M-2)\delta + 2}.
\]
Note that $E \subset \{|z-z_1| < 2R_0\} \subset |z-z_1| < 6R_0 \subset B_2$. Therefore, Lemma~\ref{lem:AT} gives
\[\notag
	cap(E,B_2) = \int_{\{|z-z_1| < 6R_0\}} (dd^c u)^n \leq  C_0 \|u\|_{B_2}^{n-1} |u(z_1)| \leq  \frac{C_0}{M \delta}.
\]
This is the desired inequality.
\end{proof}

\end{document}